\documentclass[11pt]{article}
\usepackage{amsmath,amsfonts,amsthm,amssymb,color,subfiles,fullpage}
\usepackage{tikz}
\newtheorem{theorem}{Theorem}[section]
\newtheorem{corollary}[theorem]{Corollary}
\newtheorem{lemma}[theorem]{Lemma}

\newtheorem{proposition}[theorem]{Proposition}

\newtheorem{conjecture}[theorem]{Conjecture}

\theoremstyle{definition}
\newtheorem{definition}[theorem]{Definition}
\newtheorem{remark}[theorem]{Remark}

\newdimen\pIR
\newcommand\StevesR{{\rm I\kern\pIR R}}

\def\pleq{\preccurlyeq}

\def\bvec#1{{\mbox{\boldmath $#1$}}}

\def\prob#1#2{{\mathbb{P}}_{#1}\left[ #2 \right]}

\newcommand\symExpec{\operatorname{\mathbb{E}}\displaylimits}

\def\expec#1#2{\symExpec_{#1}  #2 }

\newcommand{\E}{\mbox{{\bf E}}}

\def\defeq{\stackrel{\mathrm{def}}{=}}
\def\setof#1{\left\{#1  \right\}}

\def\trace#1{\mathrm{Tr} \left(#1 \right)}

\def\abs#1{\left|#1  \right|}

\def\norm#1{\left\| #1 \right\|}

\def\Complex#1{\mathbb{C}^{#1}}

\newcommand{\imag}[1]{\mathsf{Im} (#1)}

\def\Branden{Br\"{a}nd\'{e}n}

\newcommand{\C}{\mathbb{C}}
\newcommand{\N}{\mathbb{N}}

\newcommand{\R}{\mathbb{R}}
\newcommand{\AND}{\quad \text{and} \quad}
\newcommand{\mydet}[1]{\det\left(#1\right)}
\newcommand{\charp}[2]{\chi \left[#1 \right] \left(#2 \right)}
\newcommand{\mixed}[2]{\mu \left[#1 \right] \left(#2 \right)}

\def\above#1{\mathsf{Ab}_{#1}}
\newcommand{\tr}{\mathrm{Tr}}

\newcommand{\s}{\mathcal{S}}
\newcommand{\A}{\mathcal{A}}
\newcommand{\diag}{\mathrm{diag}}
\newtheorem{example}{Example}
\renewcommand{\P}{\mathbb{P}}
\title{The Solution of The Kadison-Singer Problem}
\author{Adam Marcus\footnote{Princeton University,
\texttt{adam.marcus@princeton.edu}, supported by NSF grant DSM-1552520.} 
\and Nikhil Srivastava\footnote{UC Berkeley,
\texttt{nikhil@math.berkeley.edu}, supported by NSF grant
CCF-1553751 and a Sloan research fellowship.}}
\date{\today}

\begin{document}
\maketitle
\tableofcontents

\begin{abstract}
These lecture notes are meant to accompany two lectures given at the CDM 2016
conference, about the Kadison-Singer Problem. They are meant to complement the
survey by the same authors (along with Spielman) which appeared at the 2014 ICM.
In the first part of this survey we will introduce the Kadison-Singer problem
from two perspectives ($C^*$ algebras and spectral graph theory) and present some
examples showing where the difficulties in solving it lie. In the second part we
will develop the framework of interlacing families of polynomials, and show how
it is used to solve the problem. None of the results are new, but we have added
annotations and examples which we hope are of pedagogical value.
\end{abstract}

\section{The Kadison-Singer Problem}
Since the Kadison-Singer Problem is a question in $C^*$ algebras, we begin by
recalling some basic definitions from that subject. Let $B(\ell_2)$ denote the
algebra of bounded operators on the complex Hilbert space $\ell_2(\N)$.  Such
operators may be identified with infinite dimensional matrices with bounded
operator norm $$\|M\|:=\sup_{x\in\ell_2}\|Mx\|.$$ For our purposes, a {\em $C^*$
algebra} is a subalgebra $\A$ of $B(\ell_2)$ which is closed in the operator
norm topology, closed under taking adjoints (hence the $C^*$), and contains the
identity. The most important example of a $C^*$ algebra in the present context
is $D(\ell_2)$, the algebra of bounded diagonal operators on $\ell_2$, which may
be identified with infinite diagonal matrices with bounded entries. Note
that $D(\ell_2)$ is a {\em maximal abelian subalgebra} of $B(\ell_2)$. 

Operator algebras were originally introduced by von Neumann as a rigorous
mathematical framework for quantum mechanics, in which bounded self-adjoint
operators play the role of physical observables (such as position, momentum,
energy). Without making a full
digression into quantum theory, we remark that the physical relevance of abelian
subalgebras of $B(\ell_2)$ is that they are generated by observables which {\em
commute}, implying that they can be measured simultaneously without being
constrained by an uncertainty principle. The physical question that motivated
the Kadison-Singer problem is roughly this: 
\begin{quote} Given a quantum system (such as an electron in a hydrogen atom) 
does knowing the outcomes of all measurements with respect to
a maximal set of commuting observables (such as the quantum numbers $n,\ell, m,
s$) {\em uniquely} determine the outcomes of all possible measurements of all
possible observables?
\end{quote}
The above is not meant to be mathematically rigorous, and we have left words
such as ``outcome'' deliberately undefined, but we remark that such an
assertion, interpreted appropriately, was believed to be true by Dirac .

We need one more notion to arrive at a mathematically precise formulation of the
question. A {\em state} on a $C^*$ algebra $\A$ is a linear functional
$\phi:\A\rightarrow\C$ with two properties: (a) $\phi(I)=1$; (2) $\phi(M^*M)\ge
0$ for every $M\in\A$. It is easy to check that the set of states on $\A$ is
convex and compact in the $w^*$ topology; let us call this set $\s(A)$. By the
Krein-Milman theorem, $\s(\A)$ is the convex hull of its extreme points, which
are the {\em pure states} on $\A$. States are supposed to correspond to physical
states of a quantum system. The only other facts we will use about states are that
they satisfy the Cauchy-Schwartz inequality:
$$|\phi(MN)|^2\le \phi(M^*M)\phi(N^*N)$$
and that $\phi(M)\le \|M\|$.

The most familiar examples of states come from unit vectors: given
any $\xi\in\ell_2$ with $\|\xi\|=1$, it is easy to see that $\rho(M):=\langle
\xi, M\xi\rangle$ satisfies (a) and (b). In finite dimensions, one can easily
show using elementary linear algebra that these are the only pure states on $B(\C^n)$. 
This is not at all the case in infinite dimensions.

The Kadison-Singer Problem asks:
\begin{quote}
	Does every pure state on $D(\ell_2)$ have a unique extension to a state
	on $B(\ell_2)$?
\end{quote}

If one restricts attention to {\em vector} pure states then the answer to the
KSP is easily seen to be yes. The difficulty stems from the fact that the set of
all pure states on $D(\ell_2)$ is substantially more complicated than this. For instance,
one can take limits of pure states with respect to non-principal ultrafilters to
produce pure states which are very different from vector states, and rather
inaccessible in concrete terms.\footnote{One could argue that this is physically
irrelevant, since non-principal ultrafilters require the axiom of choice and
cannot arise in any physical situation, and one would be right. The question
above is really a question about a certain mathematical apparatus around quantum
mechanics, rather than about the physical world itself.}

In their original paper Kadison and Singer outlined an elegant approach to proving the conjecture
without having to say too much about $\s(D(\ell_2))$. The starting point is to
observe that pure states are necessarily very well-behaved on a particular class
of operators in $D(\ell_2)$, namely {\em diagonal projections}. 
\begin{lemma}\label{lem:diagproj} If $P\in D(\ell_2)$ is a diagonal projection and $\rho$ is a pure state on
$D(\ell_2)$ then $\rho(P)=0$ or $\rho(P)=1$.\end{lemma}
\begin{proof}
	Suppose $\rho(P)=\lambda\in (0,1)$. Observe by linearity that
	$\rho(I-P)=1-\lambda$. Consider the linear functionals
	$\rho_1,\rho_2:D(\ell_2)\rightarrow \C$ defined by
	$$\rho_1(M):= \frac{1}{\lambda}\rho(PM)\qquad \rho_2(M) :=
	\frac{1}{1-\lambda}\rho((I-P)M),$$
	and observe that they are both states. But now
	$\rho=\lambda\rho_1+(1-\lambda)\rho_2$, so $\rho$ cannot be a pure
state.\end{proof}

Recall that our goal is to show that for every pure state $\rho:D(\ell_2)\rightarrow
\C$ there is a unique extension $\hat{\rho}:B(\ell_2)\rightarrow \C$. It is
clear that at least one canonical extension exists:
$$\hat{\rho}(M) := \rho(\diag(M)),$$
where $\diag(M)$ refers to the diagonal part of $M$, so to show that it is
unique we must show that any extension $\hat{\rho}$ must satisfy
$$ \hat{\rho}(M) = \rho(\diag(M)) = \hat{\rho}(\diag(M)),$$
or in other words
$$ \hat{\rho}(M-\diag(M))=0$$
for every $M\in B(\ell_2)$. This is where diagonal projections and the key
notion of a {\em paving} come in.
\begin{definition} An $\epsilon-$paving of an operator $M\in B(\ell_2)$ is a
	finite collection of diagonal projections $P_1,\ldots, P_k$ satisfying
	$P_1+\ldots+P_k=I$ and 
	$$\|P_iMP_i\|\le \epsilon \|M\|,$$
	for every $i=1,\ldots,k$.
\end{definition}
\begin{conjecture}[Paving Conjecture] For every $\epsilon>0$, every zero
	diagonal $M\in B(\ell_2)$ has an $\epsilon-$paving.\end{conjecture}
\begin{theorem}[Kadison-Singer] The Paving Conjecture implies a positive
	solution to the Kadison-Singer Problem.  \end{theorem}
\begin{proof}
Suppose $\rho$ is a pure state on $D(\ell_2)$ and $\hat{\rho}$ is an extension
of it to $B(\ell_2)$. Let $M\in B(\ell_2)$ and $N=M-\diag(M)$, and fix
$\epsilon>0$. Let $P_1,\ldots, P_k$ be an $\epsilon-$paving of $N$. Observe by
linearity that
\begin{equation}\label{eqn:pave} \hat{\rho}(N) = \sum_{i,j\le k}
\hat{\rho}(P_iNP_j).\end{equation}
By Lemma \ref{lem:diagproj} and linearity, we know that exactly one of the
projections, say $P_1$, satisfies $\hat{\rho}(P_1)=1$ and for the rest of them we have
$\hat{\rho}(P_j)=0$. 
By the Cauchy-Schwarz inequality, each term satisfies
$$ |\hat{\rho}(P_iNP_j)|\le \hat{\rho}(P_i^*P_i)\hat{\rho}(P_j^*N^*NP_j) \land
\hat{\rho}(P_i^*N^*NP_i) \hat{\rho}(P_j^*P_j).$$
Since $P^*P=P$ for any projection $P$, this implies that only the first term in
\eqref{eqn:pave} is nonzero, and we have
$$ \hat{\rho}(N)=\hat{\rho}(P_1NP_1)\le \|P_1NP_1\|\le \epsilon\|N\|,$$
by the paving property. Since $\epsilon>0$ was arbitrary we conclude that
$\hat{\rho}(N)=0$, as desired.
\end{proof}

The pleasing feature of Conjecture \ref{conj:fpc} is that it makes no mention
of pure states. The next major simplification was achieved by Anderson in 1979,
who showed that this conjecture is readily implied by a simple to state conjecture about {\em
finite} matrices.
\begin{conjecture}[Finite Paving Conjecture]\label{conj:fpc} For every $\epsilon>0$ there is a
	$k=k(\epsilon)$ such that for every $n$, every zero diagonal complex
	$n\times n$ matrix $M$ can be $\epsilon-$paved with $k$
projections.\end{conjecture}
\begin{theorem}[Anderson] The Finite Paving Conjecture implies the Infinite
Paving Conjecture.\end{theorem}
We omit the short proof, which shows that a limit of finite pavings can be used to
construct an infinite paving via an Arzela-Ascoli argument. The most important feature of this conjecture is that the
number of projections $k$ is allowed to depend only on $\epsilon$ and {\em not
on the dimension $n$}, and this is because any dependence on $n$ precludes
a limit. It is also easy to see that it is sufficient to prove the conjecture
for a single $\epsilon<1$ and constant $k$, since then any smaller $\epsilon$
can be achieved by iteration. That said, it is substantially more accessible than
the original KSP, and its statement is entirely elementary.

In the decades since Anderson's result, the paving conjecture was shown (using
various finite dimensional linear algebra arguments) to be equivalent to
several other statements about partitioning matrices or sets of vectors into
submatrices or subsets which are ``smaller'' in some appropriate sense. In
particular, the work of Casazza \cite{CEKP07} et al. shows that it is equivalent to a number
of other conjectures in various fields. A very tangible, combinatorial such
statement is the following conjecture of Weaver, which is actually a family of statements indexed by $r\in \mathbb{N}$.
\begin{conjecture}[Weaver $KS_r$]\label{conj:weaver} There are universal constants $\epsilon>0,
	\delta>0$
	such that the following holds. Suppose $v_1,\ldots, v_m\in\C^n$ are vectors
	satisfying $\sum_{i=1}^m v_iv_i^* = I$ and $\|v_i\|\le \delta$. Then
	there is a partition $[m]=T_1\cup T_2\ldots\cup T_r$ such that for every
	$j=1,\ldots,r$:
	\begin{equation}\label{eqn:xxx} \left\|\sum_{i\in T_j}
	v_iv_i^*\right\|\le 1-\epsilon.\end{equation}
\end{conjecture}
The equivalence between this conjecture and paving is obtained by passing from
paving zero diagonal matrices to paving Hermitian matrices to paving positive semidefinite matrices (by
adding a multiple of the identity) and then to paving projection matrices (via a
dilation argument)\footnote{For an exposition
of the (elementary) details of these reductions, the reader is encouraged to
consult Tao's blog post
\texttt{https://terrytao.wordpress.com/tag/kadison-singer-problem/} .}.
Dualizing the statement for projection matrices yields the family of statements
$KS_r$. It was shown in \cite{weaver} that the validity of $KS_r$ for any finite $r$
is equivalent to Kadison-Singer. 

The main result of \cite{IF2} is a strong version of $KS_r$ for every $r$:
\begin{theorem}\label{thm:partition}
Let $r > 1$ be an integer, and let $u_1, \dots, u_m \in \C^d$ be vectors such that 
\[
  \sum_{i=1}^m \expec{}{u_{i} u_{i}^{*}} = I_d
  \AND
\norm{u_i}^2 \leq \delta \text{ for all $i$.}
\]
Then there exists a partition $\{ A_1, \dots A_r \}$ of $[m]$ such that 
\[
\norm{ \sum_{i \in A_j} {u_i} {u_i}^{*}} \leq \frac{1}{r}\left(1+\sqrt{r \delta}\right)^{2}
\]
\end{theorem}
Since the outer products of the vectors sum to the identity, the best one could hope for is to be able to split the vectors into $r$ groups such that each was {\em exactly} $(1/r)I$.
Hence Theorem~\ref{thm:partition} guarantees that for any vectors $v_i$, one can
get within a factor of $\left(1+\sqrt{r \delta}\right)^{2}$ of the best one
could get with the best possible $u_i$.

\section{Spectral Graph Theory}
In this section we describe a different, more recent story which leads to the same core
problem $KS_2$, and which is in fact how the present authors were introduced to
this problem. The question we consider is: 

\begin{quote}
Given a finite undirected graph, can it be approximated by a graph with very few
edges?
\end{quote}

The answer to this question of course depends on what we mean by approximate,
and this is where the Laplacian operator comes in.  Recall that the discrete Laplacian of a weighted graph $G=(V,E,w)$ may be defined as the following sum of
  rank one matrices over the edges:
\[
  L_{G} = \sum_{(a,b) \in E} w_{(a,b)} (e_{a} - e_{b}) (e_{a} - e_{b})^{T}.
\]
In the unweighted $d-$regular case, it is easy to see that $L=dI-A$, so the
eigenvalues of the Laplacian are just $d$ minus the eigenvalues of the adjacency
  matrix.
The Laplacian matrix of a graph always has an eigenvalue of $0$; this is a 
  trivial eigenvalue, and the corresponding eigenvectors are the constant vectors.

Following Spielman and Teng, we say that two graphs $G$ and $H$ on the same
  vertex set $V$ are {\em spectral approximations} of each other if their
  Laplacian quadratic forms multiplicatively approximate each other:
  $$ \kappa_1\cdot x^TL_Hx\leq x^TL_Gx\leq \kappa_2 \cdot x^TL_Hx\qquad\forall x\in\R^V,$$
for some approximation factors $\kappa_1,\kappa_2>0$. We will write this as
	$$ \kappa_1\cdot L_H \pleq L_G\pleq \kappa_2\cdot L_H,$$
where $A\pleq B$ means that $B-A$ is positive semidefinite, i.e., $x^T(B-A)x\ge
  0$ for every $x$.

The complete graph on $n$ vertices, $K_{n}$, is the graph with an edge of weight $1$ between every pair of vertices.
All of the eigenvalues of $L_{K_{n}}$ other than $0$ are equal to $n$.
If $G$ is a $d$-regular Ramanujan graph \cite{LPS}, then $0$ is the trivial 
  eigenvalue of its Laplacian matrix, $L_{G}$,
  and all of the other eigenvalues of $L_{G}$ are between $d - 2 \sqrt{d-1}$ and $d + 2 \sqrt{d-1}$.
After a simple rescaling, this allows us to conclude that
\[
  (1 - 2 \sqrt{d-1}/d) L_{K_{n}} 
  \pleq (n/d) L_{G} 
  \pleq  (1 + 2 \sqrt{d-1}/d)  L_{K_{n}} .
\]
So, $(n/d) L_{G}$ is a good approximation of $L_{K_{n}}$.

Batson, Spielman and Srivastava proved that every weighted graph has an approximation that is almost this good.
\begin{theorem}[\cite{BSS}]\label{thm:BSSgraphs}
For every $d>1$
  and every weighted graph $G=(V,E,w)$ on $n$ vertices,
  there exists a weighted graph $H=(V,F,\tilde{w})$ with $\lceil d(n-1)\rceil$
  edges that satisfies:
\begin{equation}\label{eqn:BSSgraph}
\left(1-\frac{1}{\sqrt{d}}\right)^2 L_G  \pleq L_{H} \pleq
\left(1+\frac{1}{\sqrt{d}}\right)^2 L_{G}.
\end{equation}
\end{theorem}

However, their proof had very little to do with graphs.
In fact, they derived their result from the following theorem about sparse weighted
  approximations of sums of rank one matrices.
\begin{theorem}[\cite{BSS}]\label{thm:BSSmatrices}
Let $v_{1}, v_{2}, \ldots , v_{m}$ be vectors in $\R^{n}$ with
\[
  \sum_{i} v_{i} v_{i}^{T} = V.
\]
For every $\epsilon \in (0,1)$,
  there exist non-negative real numbers
  $s_{i}$
  with
\[
|\{i : s_{i} \not = 0\}|
\leq 
\lceil n / \epsilon^{2}\rceil
\]
so that
\begin{equation}\label{eqn:BSSvecs}
(1-\epsilon)^{2} V 
  \pleq \sum_{i} s_{i} v_{i} v_{i}^{T}
  \pleq (1+\epsilon)^{2} V.
\end{equation}
\end{theorem}
Taking $V$ to be a Laplacian matrix written as a sum of outer products and
  setting $\epsilon=1/\sqrt{d}$ immediately yields Theorem \ref{thm:BSSgraphs}.

Theorem \ref{thm:BSSmatrices} is very general and turned out to be useful in a
  variety of areas including graph theory, numerical linear algebra, and metric
  geometry (see, for instance, the survey of Naor \cite{naor2011sparse}).
One of its limitations is that it provides no guarantees on the weights $s_i$
  that it produces, which can vary wildly.  
So it is natural to ask: is there a version of Theorem \ref{thm:BSSmatrices} in
  which all the weights are the same?

This may seem like a minor technical point, but it is actually a fundamental
  difference.
In particular, Gil Kalai observed that the statement of Theorem~\ref{thm:BSSmatrices} with
  $V=I$ is similar to Weaver's Conjecture.
It turns out that the natural unweighted variant of it is essentially {\em the
  same} as Weaver's conjecture.

To make the connection, let us go back to the setting of $KS_2$ and observe that for any
partition of a given set of vectors $v_1,\ldots,v_m$ we have:
\[
   \sum_{i \in S_{1}} v_{i} v_{i}^{*} + \sum_{i \in S_{2}} v_{i} v_{i}^{*}  = I,
\]
so that condition \eqref{eqn:xxx}
  is equivalent to
\[
  \epsilon I \pleq  {\sum_{i \in S_{1}} v_{i} v_{i}^{*}} \pleq (1-\epsilon) I.
\]
Thus, choosing a subset of the weights $s_i$ to be non-zero in Theorem
  \ref{thm:BSSmatrices} is similar to choosing the set $S_{1}$.
The essential difference is that Conjecture~\ref{conj:weaver} assumes a bound on
the lengths of the vectors $v_i$ and in return requires the stronger conclusion
  that all of the $s_i$ are either $0$ or $1$.
It is easy to see that long vectors are an obstacle to the existence of a good
  partition; an extreme example is provided by considering an orthonormal basis
  $e_1,\ldots,e_n$.
Weaver's conjecture asserts that this is the only obstacle.

\section{Two Examples and their Expected Characteristic Polynomials}
In this section we discuss two key examples which highlight the difficulties in
solving Weaver's problem using familiar combinatorial and random matrix
techniques.

\begin{example}[Diagonal Case]\label{ex:diag} Let $\delta>0$ and $m=n/\delta$, and let $v_1,\ldots,v_m$
consist of $1/\delta$ copies each of $\sqrt{\delta}e_1, \sqrt{\delta}e_2,\ldots,
\sqrt{\delta}e_n$, where $e_i$ are the standard basis vectors in $\C^n$. Then it is clear
that $\|v_i\|^2 = \delta$ for every $i$ and 
$$\sum_{i\le m} v_iv_i^*= \frac{1}{\delta}\sum_{i\le n} \delta e_ie_i^*=I.$$
\end{example}
It is not hard to find a balanced partition in this example: for each standard
basis vector $e_i$, simply divide the copies of that vector into subsets of
almost equal size. Note that this simple {\em deterministic} strategy crucially requires
knowing that the given vectors can be split up into $n$ groups, each of which is
a a one-dimensional instance of $KS_2$. Also note that it would not be as clear
how to proceed if one were to (say) add a small amount of noise to each vector
--- a clustering approach might still work, but would be somewhat nontrivial.

On the other hand, balanced partitions of these vectors are exponentially rare.
To see this, consider a uniformly random partition of $v_1,\ldots,v_m$ into
$T_1\cup T_2$. Then $T_1$ is a random subset of $[m]$, containing each $v_i$
with probability $1/2$. Thus, for any $i$ the probability that all copies of
$e_i$ appear in $T_1$ is $(1/2)^{1/\delta}$, and the probability that this does
{\em not} happen for all $i=1,\ldots,n$ is:
$$ \left(1-2^{-1/\delta}\right)^n \approx \exp(-n2^{-1/\delta}),$$
which is exponentially small unless $\delta\le 1/\log(n)$. A similar probability is
obtained even if we consider random balanced partitions with $|T_1|=|T_2|$, but
we omit the details.

The second example exhibits exactly the opposite kind of behavior and is given
by random vectors.
\begin{example}[Random Case]\label{ex:random} Let $\delta>0$ and $m=n/\delta$ and let $v_1,\ldots,v_m\in\R^n$ be
	i.i.d. random Gaussian vectors scaled so that
	$$ \E \|v_i\|^2 =\delta.$$
	By standard concentration inequalities (see e.g. \cite{barvinok2005math}) we have:
	$$ \P [ \|v_i\|^2  > (1+t)\delta] \le \exp(-t^2n/4),$$
	which implies by a union bound that
	$$ \max_i \|v_i\|^2 \le (1+o(1))\delta\qquad \textrm{with probability at
	least}\quad 1-\exp(-cn),$$
	as long as $m=\exp(o(n))$. Moreover, by well-known properties of
	rectangular Gaussian random matrices (e.g., Sections 5.3 and 5.4 of
\cite{vershynin2010introduction}) the eigenvalues of
	$$ V=\sum_{i\le m} v_iv_i^T $$
	are contained in the interval $[(1-\sqrt{\delta}-o(1))^2,
	(1+\sqrt{\delta}+o(1))^2]$ with exponentially good probability.
	Thus, the vectors $w_i:=V^{-1/2}v_i$ satisfy the conditions of $KS_2$
	with constant at most $5\delta$ whenever (say) $\delta<1/4$.
\end{example}
It is not clear how to deterministically partition a typical instance of Example
2 --- in particular, pairs of vectors are not orthogonal and they do not naturally decompose into
lower-dimensional instances of $KS_2$. 

However, in contrast to the previous example, a random partition works very well
here. If we take $T_1$ to be a random subset of $[m]$ of size $m/2$, then 
$$V_1=\sum_{i\in T_1} v_iv_i^T$$
is itself a Wishart matrix whose expectation is $I/2$. Again by the Bai-Yin
theorem, we conclude that the eigenvalues of $V_1$ are contained in $[\frac12
(1-\sqrt{2\delta}+o(1))^2, \frac12 (1+\sqrt{2\delta}-o(1))^2]$ with exponentially high
probability. The same is true for $T_2=[m]\setminus T_1$, so we conclude that a
random partition is balanced with high probability.

The difficulty of $KS_2$ arises from the fact that there are no tools which
readily handle both examples and the various possible combinations of them. 

The following well-known result in random matrix theory can be used to analyze a
random partition, by taking $A_i=v_iv_i^T$ and taking $T_1$ to be all $i$ such
that $\epsilon_i=+1$. 
\begin{theorem}[Matrix Chernoff \cite{tropp2012user}]
Given random Hermitian matrices $A_1\ldots A_m\in\C^{n\times n}$ and independent Bernoulli
signs $\epsilon_1,\ldots,\epsilon_m$, we have
$$\mathbb{P}\left[ \|\sum_i \epsilon_i A_i\|\ge t\right]\leq n\cdot
\exp(-\frac{t^2}{2 \|\sum_i A_i^2\|}).$$
\end{theorem}
Note that for an instance of $KS_2$ we have $\sum_i A_i^2 = \sum_i
\|v_i\|^2v_iv_i^T\le \delta I$, so the above probability is less than one when
$t$ is a small constant and $\delta\le c/\log(n)$, yielding a balanced
partition. The tightness of this result is witnessed by
Example 1, which shows that for larger $\delta$ the probability of being
unbalanced quickly approaches one.  Unfortunately, a result in which $\delta$
depends on $n$ is insufficient for Kadison-Singer (since we will take a limit of
pavings as $n\rightarrow \infty$) and also for various graph theory
applications. 

There are of course many other results about random matrices which can be used
to analyze specific families of instances. However, results which establish
dimension-free bounds (without the fatal $\log(n)$ factor introduced by Matrix
Chernoff) typically rely on high symmetry assumptions about the vectors (such as
i.i.d. entries) or on strong geometric regularity properties (such as
log-concavity), which are far too restrictive for the general case.
Thus, new ideas are required.
\subsection*{Expected Characteristic Polynomials}
The Method of Interlacing Families is a way of analyzing certain random matrices
which is oblivious to the diagonal/random dichotomy above, i.e., it is able to
provide a uniform dimension-free bound on both cases and on everything in
between. Unlike most results about random matrices, it provides estimates on the
eigenvalues which hold with {\em exponentially small} but nonetheless nonzero
probability. As such the method is not really probabilistic in nature, but the
language of probability theory provides a convenient notation.

The central idea is to access the distribution of the eigenvalues of a random
matrix via its {\em expected characteristic polynomial}:
$$\E \chi(A) = \E \det(xI-A).$$

Before describing the general approach, let's examine the expected
characteristic polynomials corresponding to Examples 1 and 2 above. 

\begin{example}[$\E\chi$ for the Diagonal Case] Let $v_1,\ldots,v_m$ be as in
Example \ref{ex:diag} with $m=kn$ where $k=1/\delta$ is an integer. Let 
$$A=\sum_{i\le m} b_iv_iv_i^T,$$
where $b_1,\ldots,b_m$ are i.i.d. random variables each $0$ with probability $1/2$ and $1$ otherwise, corresponding
to a random subset of the vectors. Certainly $\E A=I/2$, and the $A$ is diagonal
with independent entries $A_1,\ldots,A_n$ indicating the number of times the
vector $\sqrt{\delta}e_i$ is chosen. We now have
\begin{align*}
\E\det(xI-A) &= \E (x-A_1)(x-A_2)\ldots (x-A_n)\\
&=\prod_{i\le n} \E (x-A_i)\quad\textrm{since the $A_i$ are independent}
\\&=(x-1/2)^n.\end{align*}
There are two interesting things about this calculation. The first is that the expected
characteristic polynomial is real-rooted, which is in general not the case since
real-rootedness is not necessarily preserved under taking sums. The second is
that the roots reflect the behavior of (one half of) the ideal balanced
partition: the bad allocations with $A_i$ that are
too large or small seem to have have ``cancelled out''.
\end{example}

\begin{example}[$\E\chi$ for the Random Case] Let $v_1,\ldots,v_m$ be Gaussian random vectors
with norm $\E \|v_i\|^2 = \delta$ and let $A=\sum_{i\le m/2}v_iv_i^T$ be the
empirical covariance matrix of half of them. Observe that for any matrix $M$ and
a single Gaussian random vector $v$ with $\E \|v\|^2=\delta$ we have:
\begin{align*}
\E \det(xI-M-vv^T) &= \det(xI-M)\cdot \E (1-v^T(xI-M)^{-1}v)
\\&= \det(xI-M) (1-\mathrm{Tr}((xI-M)^{-1}\bullet \E vv^T))
\\&= \det(xI-M) -\delta \det(xI-M)\cdot \mathrm{Tr}(xI-M)^{-1}
\\&= \left(1-\delta \frac{d}{dx}\right)\det(xI-M)
\end{align*}
Thus, adding a random rank one Gaussian outer product corresponds to subtracting
off a multiple of the derivative of the characteristic polynomial. Since the
$v_i$ in our example are independent, we may apply this fact inductively to
conclude that
\begin{equation}\label{eqn:laguerre} \E\det(xI-A)
=\left(1-\delta\frac{d}{dx}\right)^{m/2} x^n.\end{equation}

We now observe that if a polynomial $f(x)=\prod_{i\le n} (x-\lambda_i)$ has real roots, then $f(x)-cf'(x)$
also has real roots for every real $c$; the reason is that $f(x)-cf'(x)=0$
precisely when 
\begin{equation}\label{eqn:logderiv}
\sum_{i=1}^n\frac{1}{x-\lambda_i}=f'(x)/f(x)=1/c.\end{equation}
By examining the behavior of this rational function between its poles and
noting that every multiple root of $f$ is also a root of $f'$ with one less
multiplicity,  we see that the number
of solutions to \eqref{eqn:logderiv} is equal to the degree of $f$.

Thus, we conclude that \eqref{eqn:laguerre} has real roots. But more is true: it
turns out that these polynomials are exactly equal to certain orthogonal
polynomials known as the {\em associated Laguerre polynomials}, whose roots have
been studied in great detail (see \cite{MSSICM} Section 3.2 for more details). This
connection implies that the roots of $\E \det(xI-A)$ are contained in the inverval $[\frac12
(1-\sqrt{2\delta})^2, \frac12 (1+\sqrt{2\delta})^2]$, which is precisely what is
expected for a random partition in \ref{ex:random}.
\end{example}

Thus, in both of the extreme cases, the expected characteristic polynomial is
real-rooted and captures the behavior of the kind of partition that we want ---
greedy in the diagonal case, and random in the random case. 
It turns out that this is not an accident and for a large class of random
matrices the expected characteristic polynomial {\em always}
has real roots, a property which can be used to relate the roots to the
distribution of the eigenvalues of the matrix itself. Then, tools from the
analytic theory of polynomials can be used to bound the roots, and thereby
obtain information about the eigenvalues.

The main theorem produced by this approach is the following:
\begin{theorem}\label{thm:general}
Let $\epsilon > 0$ and let
  $v_1, \dots, v_m$ be independent random vectors in $\C^d$ with finite support such that 
\begin{equation}\label{eqn:mainSum}
\sum_{i=1}^m \expec{}{v_{i} v_{i}^{*}} = I_d,
\end{equation}
and
\[
\expec{}{ \norm{v_{i}}^{2}} \leq \epsilon, \text{for all $i$.}
\]
Then 
\[
\prob{}{ \norm{\sum_{i=1}^m v_i v_i^{*}} \leq (1 + \sqrt{\epsilon})^2 } > 0
\]
\end{theorem}
The above theorem compares favorable to what is yielded by the Matrix Chernoff
bound, which is that  $\|\sum_{i=1}^mv_iv_i^*\|\le C(\epsilon)\cdot\log n$ with high
  probability.
Here we are able to control the deviation at the much smaller
  scale $(1+\sqrt{\epsilon})^2$, but only with nonzero probability.

In the remainder of this document, we show how to prove this theorem.

\section{Interlacing Families}\label{sec:interlacing}
Our proofs inherently rely on bounding the largest eigenvalue of various matrices.
One key idea in our methods for obtaining such bounds is the use of characteristic polynomials in our analysis.
Since the eigenvalues of a matrix $A$ are exactly the roots of its characteristic polynomial $\mydet{xI - A}$, this does not seem to gain us any leverage.
The leverage comes, however, when we replace {\em random} matrices with {\em random} characteristic polynomials.

On the surface this may seem like an odd idea. 
In general, the roots of an average of real rooted polynomials are not necessarily related in any way to the roots of the original collection.
However, there are situations where this works quite well.
\begin{lemma}\label{lem:separate}
Let $p_1, \dots, p_k$ be polynomials and $[s, t]$ an interval such that
\begin{itemize}
\item each $p_i(s)$ has the same sign (or is $0$)
\item each $p_i(t)$ has the same sign (or is $0$)
\item each $p_i$ has exactly one real root in $[s, t]$.
\end{itemize}
Then $\sum_i p_i$ has exactly one real root in $[s, t]$ and it lies between the roots of some $p_a$ and $p_b$.
\end{lemma}
\begin{proof}
This is illustrated in Figure~\ref{fig:sep} (the blue line is the average of the red ones).
Let $p(x) = \sum_i p_i(x)$ and without loss of generality, assume $p_i(s) \geq 0$ for all $i$.
Then $p(x) \geq 0$ and, since each polynomial switches signs somewhere in the interval $[s, t]$ we have $p(t) \leq 0$.
By continuity, there must be a point $r \in [s, t]$ at which $p(r) = 0$.

Now if we look at the value of each $p_i$ at the point $r$, we know the values must add up to $0$.
Hence there exist polynomials $p_a$ and $p_b$ such that  $p_a(r) \leq 0 \leq p_b(r)$ and these will have roots which are smaller (larger) than $r$ (respectively).
\end{proof}

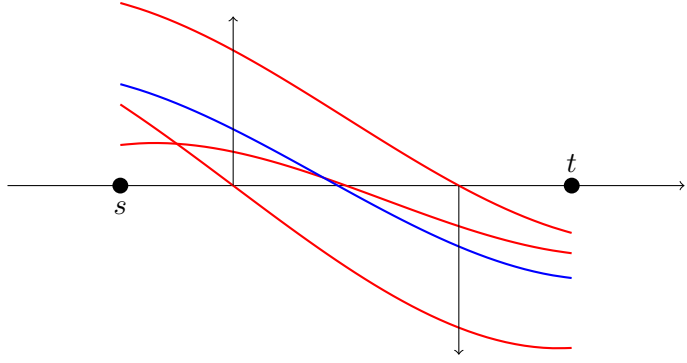
\begin{figure}[h]
 \begin{center}
\begin{tikzpicture}[scale=1.5]

\draw[->] (0,0) -- (6,0);
\coordinate (s) at (1,0);
\coordinate (t) at (5,0);

\draw[domain=1:5,smooth,variable=\x,red,thick]  plot ({\x},{0.03*(\x )*(\x - 3)*(\x - 7)});
\draw[domain=1:5,smooth,variable=\x,red,thick]  plot ({\x},{0.03*(\x + 3)*(\x - 2)*(\x - 7)});
\draw[domain=1:5,smooth,variable=\x,red,thick]  plot ({\x},{0.03*(\x + 2)*(\x - 4)*(\x - 7)});
\draw[domain=1:5,smooth,variable=\x,blue,thick]  plot ({\x},{0.03*(\x * \x - (4/3) *\x - (14/3))*(\x - 7)});

\fill (s) circle (2pt);
\fill (t) circle (2pt);

\node (ls) at (1, -0.2) {$s$};
\node (lt) at (5, 0.2) {$t$};

\draw[->] (2, 0) -- (2, 1.5);
\draw[->] (4, 0) -- (4, -1.5);

\end{tikzpicture}
\end{center}
\caption{Picture of Lemma~\ref{lem:separate}.}
\label{fig:sep}
\end{figure}

Lemma~\ref{lem:separate} asserts that, as long as our collection of polynomials
has its roots bunched together inside disjoint intervals, then the sum of
the polynomials is real rooted and one can compare the roots of the sum to the
individual polynomials.  To characterize the collections of polynomials for
which this holds, we recall the definition of interlacing polynomials.

\begin{definition}\label{def:interlacing}
We say that a real rooted
  polynomial $g(x) = \alpha_{0} \prod_{i=1}^{n-1} (x - \alpha_{i})$ \emph{interlaces} a
  real rooted polynomial 
  $f(x) = \beta_{0} \prod_{i=1}^{n} (x - \beta_{i})$ if
\[
  \beta_{1} \leq \alpha_{1} \leq \beta_{2} \leq \alpha_{2} \leq \dotsb \leq
  \alpha_{n-1}\leq \beta_{n} 
\]
We say that $g (x)$ \emph{strictly interlaces} $f (x)$ if all of these inequalities are strict.
We say that polynomials $f_{1} (x), \dotsc , f_{k} (x)$ have a \emph{common interlacing}
  if there is a polynomial $g (x)$ so that $g (x)$ interlaces $f_i (x)$ for each $i$.
\end{definition}

In the event that a collection of polynomials $f_1, \dots f_k$ has a common
interlacer $g$, the roots of $g$ separate the roots of the $f_i$ in exactly the
way necessary for Lemma~\ref{lem:separate} to hold.  This leads to the following
corollary:

\begin{corollary}\label{lem:interlacing}
Let $f_{1}, \dotsc , f_{k}$ be polynomials of the same degree that are real-rooted and have positive leading coefficients.
Define
\[
  f_{\emptyset} = \sum_{i=1}^{k} f_{i}.
\]
If $f_{1}, \dotsc , f_{k}$ have a common interlacing,
  then
  there exists an $i$ so that the largest root of $f_{i}$ is at most the largest root of
  $f_{\emptyset}$.
\end{corollary}

The hope would be to apply Lemma~\ref{lem:interlacing} to the collection of polynomials defined in Section~\ref{sec:mixed}.
These polynomials, however, do not have a common interlacing.
Instead, we will need to use Lemma~\ref{lem:interlacing} inductively on subcollections of these polynomials that do have a common interlacing.
This inspires the following definition from \cite{IF1}:

\begin{definition}\label{def:family}
Let $S_{1}, \dots , S_{m}$ be finite sets and for every assignment $s_{1}, \dots, s_{m} \in S_{1} \times \dots \times S_{m}$
  let $f_{s_{1}, \dots , s_{m}} (x)$ be a real-rooted degree $n$ polynomial with positive leading coefficient.
For a partial assignment $s_1, \dots, s_k \in S_1 \times \ldots \times S_k$ with $k < m$, define
\[
  f_{s_{1},\dots , s_{k}} \defeq
\sum_{s_{k+1} \in S_{k+1}, \dots , s_{m} \in S_{m}}
  f_{s_{1}, \dots ,s_{k}, s_{k+1}, \dots , s_{m}},
\]
as well as
\[
  f_{\emptyset} \defeq \sum_{s_{1} \in S_{1}, \dots , s_{m} \in S_{m}}
  f_{s_{1}, \dots , s_{m}}.
\]

We say that the polynomials $\{f_{s_{1}, \dots , s_{m}} \}$
  form an \textit{interlacing family} if for all $k=0,\ldots, m-1$, and all
  $s_{1}, \dots, s_{k} \in S_{1} \times \dots \times S_{k}$,
  the polynomials
\[
  \{f_{s_{1}, \dots , s_{k},t}\}_{t\in S_{k+1}}
\]
have a common interlacing.
\end{definition}

Given an interlacing family, one can form a tree of partial assignment
polynomials with $f_{\emptyset}$ at the top (we avoid saying ``root'' since we
already use it as a synonym for ``zeroes'') and where each polynomial $f_{s_1,
\dots, s_k}$ will have the collection of polynomials $\{ f_{s_1, \dots, s_k, t}
\}_{t \in S_{k+1}}$ as its children.  The idea, then, will be to apply
Lemma~\ref{lem:interlacing} iteratively as one walks down the tree (see
Figure~\ref{fig:tree}).

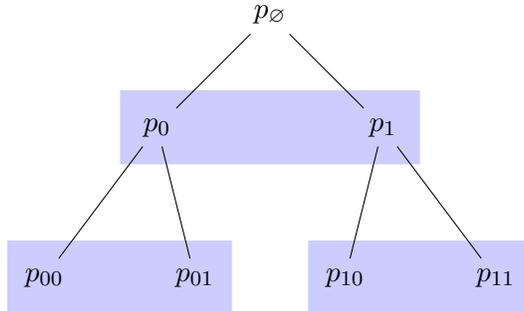
\begin{figure}[h]
\begin{center}
\begin{tikzpicture}

\draw[fill=blue!20,draw=white] (0.5,0) rectangle  (3.5,1);
\draw[fill=blue!20,draw=white] (4.5,0) rectangle  (7.5,1);

\draw[fill=blue!20,draw=white] (2,2) rectangle  (6,3);

\draw(1,0.5) node (p1) {$p_{00}$};
\draw(3,0.5) node (p2) {$p_{01}$};
\draw(5,0.5) node (p3) {$p_{10}$};
\draw(7,0.5) node (p4) {$p_{11}$};

\draw (2.5, 2.5) node (p12) {$p_{0}$};
\draw (5.5, 2.5) node (p34) {$p_{1}$};

\draw[-] (p12) -- (p1);
\draw[-] (p12) -- (p2);
\draw[-] (p34) -- (p3);
\draw[-] (p34) -- (p4);

\draw (4, 4) node (p1234) {$p_{\varnothing}$};
\draw[-] (p1234) -- (p12);
\draw[-] (p1234) -- (p34);

\end{tikzpicture}
\end{center}
\caption{A tree of partial assignment polynomials. The purple blocks denote subsets of polynomials that have a common interlacer.}
\label{fig:tree}
\end{figure}

\begin{theorem}\label{thm:interlacing}
Let $S_{1}, \dots , S_{m}$
  be finite sets and let
 $\setof{f_{s_{1}, \dots , s_{m}} }$ be an interlacing family of polynomials.
Then, there exists some $s_{1},\dots , s_{m} \in S_{1} \times \dots \times S_{m}$
  so that the largest root of
  $f_{s_{1}, \dots , s_{m}}$ is less than the largest root of $f_{\emptyset}$.
\end{theorem}
\begin{proof}
From the definition of an interlacing family, we know that the polynomials $\{ f_t\}$ for ${t \in S_1}$
  have a common interlacing and that their sum is $f_{\emptyset}$.
So, Lemma~\ref{lem:interlacing} tells us that one of the polynomials has largest root at most the largest root of $f_{\emptyset}$.
We now proceed inductively.
For any $s_{1}, \dots , s_{k}$, we know that the polynomials
  $\{f_{s_{1}, \dots , s_k, t}\}$ for $t \in S_{k+1}$ have a common interlacing and
  that their sum is $f_{s_{1}, \dots , s_{k}}$.
So, for some choice of $t$ (say $s_{k+1}$) the largest root of the polynomial
  $f_{s_{1}, \dots , s_{k+1}}$
  is at most the largest root of
  $f_{s_{1}, \dots , s_{k}}$.
\end{proof}

Our first goal will be to prove that the characteristic polynomials of sums of
independent rank one random matrices form an interlacing family.
According to Definition~\ref{def:family}, this requires establishing the existence of certain common interlacings.
We will do this using the fact that common interlacings are equivalent to
real-rootedness statements, a result which seems to have been discovered a number of times.
The following appears as Theorem~$2.1$ of Dedieu~\cite{Dedieu}, (essentially) as
Theorem~$2'$ of Fell~\cite{Fell}, and as (a special case of) Theorem 3.6 of Chudnovsky and
Seymour~\cite{ChudnovskySeymour}.

\begin{lemma}\label{lem:Fisk}
Let $f_1,\ldots, f_k$ be (univariate) polynomials of the same degree with
positive leading coefficients. Then $f_1,\ldots, f_k$ have a common interlacing if
and only if $\sum_{i=1}^k\lambda_i f_i$ is real rooted for all convex
combinations $\lambda_i \geq 0, \sum_{i=1}^k\lambda_i=1$.
\end{lemma}
We establish the necessary real-rootedness statements in Section 5.

Our second goal will be to bound the largest root of the corresponding
$f_{\emptyset}$ and then use Theorem~\ref{thm:interlacing} to assert the
existance of some polynomial in the original collection that has largest root
smaller than $f_{\emptyset}$.

The analysis will benefit from a slight generalization of interlacing that allows for polynomials to have the same degree.
Let $f$ and $g$ be real rooted polynomials and let $r_f$ and $r_g$ be their smallest roots (respectively).
\begin{definition}
We say that $f$ {\em subinterlaces} $g$ (written $f \ll g$) if either 
\begin{enumerate}
\item $f$ interlaces $g$, or
\item $r_f \leq r_g$ and $f(x) / (x - r_f)$ interlaces $g$.
\end{enumerate}
\end{definition}

In particular, we will use the following theorem of Hermite, Kakeya, and Obreschkoff:

\begin{theorem}{Hermite-Kakeya-Obreschkoff}\label{thm:HKO}
Let $f$ and $g$ be real rooted polynomials.
Then $af + bg$ is real rooted for all $a,b \in \R$ if and only if $f \ll g$ or $g \ll f$.
\end{theorem}

\section{The Mixed Characteristic Polynomial}\label{sec:mixed}

In this section, we obtain a useful formula for the expected characteristic
polynomials which are relevant to the proof of Theorem \ref{thm:general}, and
show that these polynomials are always real-rooted, which is crucial to the
interlacing method.

We begin by recording some well-known facts from linear algebra.
For a Hermitian matrix $M \in \Complex{d \times d}$ we write the characteristic polynomial of $M$ 
  in a variable $x$ as
\[
  \charp{M}{x} = \mydet{xI - M}.
\]
The following lemma is sometimes known as the {\em matrix determinant lemma} or {\em rank-$1$ update formula}.

\begin{lemma}\label{lem:rank1update}
If $A$ is an invertible matrix and $u,v$ are vectors, then
\[
\mydet{A + uv^*} = \mydet{A} (1 + v^* A^{-1} u)
\]
\end{lemma}

We will utilize Jacobi's formula for the derivative of the determinant of a matrix.
\begin{theorem}\label{thm:jacobi}
If $A$ and $B$ are matrices of the same dimensions and $A$ is invertible, then 
\[
\partial_t \mydet{A + tB} \Big|_{t=0} = \mydet{A}\trace{A^{-1} B}.
\]
\end{theorem}

Using the previous two results, we have the following easy corollary.

\begin{corollary}\label{cor:jameslee}
For an invertible matrix $A$ and random vector $v$, we have
\[
\expec{}{\mydet{A - vv^*}} 
= (1 - \partial_t) \mydet{A + t \expec{}{vv^*}}\Big|_{t=0}
\]
\end{corollary}
\begin{proof}
By Lemma~\ref{lem:rank1update}, we have
\begin{align}
\expec{}{\mydet{A - vv^*}}
&= \expec{}{\mydet{A}(1 - v^*A^{-1}v)} \notag
\\&= \expec{}{\mydet{A}(1 - \trace{A^{-1}vv^*})} \notag
\\&= \mydet{A} - \mydet{A}\expec{}{\trace{A^{-1}vv^*}} \label{eq:jameslee}
\end{align}
On the other hand, by Theorem~\ref{thm:jacobi}, we have
\begin{align*}
(1 - \partial_t) \mydet{A + t \expec{}{vv^*}}\Big|_{t=0}
&= \mydet{A + t \expec{}{vv^*}}\Big|_{t=0} - \mydet{A}\trace{A^{-1}\expec{}{vv^*}}
\\&= \mydet{A} - \mydet{A}\trace{A^{-1}\expec{}{vv^*}}
\end{align*}
which is the same as (\ref{eq:jameslee}) by switching the order of summation in the expectation/trace.
\end{proof}

Let $v_{1}, \dots , v_{m}$ be independent random column vectors in $\Complex{d}$ with finite support.
For each $i$, let $A_{i} =  \expec{}{v_{i} v_{i}^{*}}$.
Then,
\begin{theorem}\label{thm:mixed1}
\begin{equation}\label{eqn:mixed1}
\expec{}{\charp{\sum_{i=1}^{m} v_{i} v_{i}^{*}}{x}}
= 
\left(\prod_{i=1}^{m} 1 - \partial_{z_{i}} \right) 
\mydet{x I + \sum_{i=1}^{m} z_{i} A_{i}}
\bigg|_{z_{1} = \dots = z_{m} = 0}.
\end{equation}
\end{theorem}
\begin{proof} 
For a positive definite matrix $M$, set 
\[
a_k(M) = \expec{}{\mydet{M - \sum_{i=1}^{k} v_{i} v_{i}^{*}}}
\]
and
\[
b_k(M)
=\left(\prod_{i=1}^{k} 1 - \partial_{z_{i}} \right) 
\mydet{M + \sum_{i=1}^{k} z_{i} A_{i}}
\bigg|_{z_{1} = \dots = z_{k} = 0}.
\]
We will prove by induction (on $k$) that $a_k(M) = b_k(M)$.
As the base case, we have 
\[
a_0(x) = \expec{}{\mydet{M}} = \mydet{M} = b_0(x)
\]
so we may assume $a_i(M) = b_i(M)$ for all $i < k$.
Now Corollary~\ref{cor:jameslee} implies
\begin{align*}
a_k(M) 
&= \expec{}{\mydet{M - \sum_{i=1}^{k} v_{i} v_{i}^{*}}}
\\&= \expec{v_1, \dots, v_{k-1}}{\expec{v_k}{\mydet{M - \sum_{i=1}^{k-1} v_{i} v_{i}^{*} - v_{k} v_{k}^{*}}}}
\\&= \expec{v_1, \dots, v_{k-1}}{(1 - \partial_{z_k})\mydet{M - \sum_{i=1}^{k-1} v_{i} v_{i}^{*} + z_k A_k}}\bigg|_{z_k = 0}
\end{align*}
and so by the inductive hypothesis (with $M' = M + z_k A_k$, which for $z_k$ sufficiently close to $0$ is still positive definite), we get
\begin{align*}
a_k(M) 
&= (1 - \partial_{z_k})\left(\prod_{i=1}^{k-1} 1 - \partial_{z_{i}} \right) \mydet{M' + \sum_{i=1}^{k-1} z_{i} A_{i} }\bigg|_{z_1=\dots = z_{k-1} = 0} \bigg|_{z_k = 0}
\\&=\left(\prod_{i=1}^{k} 1 - \partial_{z_{i}} \right) \mydet{M + \sum_{i=1}^{k} z_{i} A_{i} }\bigg|_{z_1=\dots = z_{k} = 0} 
\\&=b_k(M).
\end{align*}
Hence $a_k(M) = b_k(M)$ for all positive definite $M$.
In particular, $a_m(xI) = b_m(xI)$ for $x > 0$.
But $a_m(xI)$ and $b_m(xI)$ are finite degree polynomials, so equality on any interval implies equality everywhere.
\end{proof}

In particular, the above formula shows that the expected characteristic polynomial is a function of the
  covariance matrices $A_i$. 
We call this polynomial the \textit{mixed characteristic polynomial} of
  $A_{1}, \dots , A_{m}$, and denote it by $\mixed{A_{1}, \dots , A_{m}}{x}$.

To see that these polynomials are real-rooted, we draw on the theory of
multivariate real stable polynomials.

\subsubsection*{Stable Polynomials}
For a complex number $z$, let $\imag{z}$ denote its imaginary part.
We recall that a polynomial $p(z_{1}, \dots , z_{m}) \in \Complex{}[z_{1}, \dots , z_{m}]$
  is \textit{stable} if whenever $\imag{z_{i}} > 0$ for all $i$,
  $p(z_{1}, \dots , z_{m}) \not = 0$.
A polynomial $p$ is \textit{real stable} if it is stable and all of its coefficients are real.
A univariate polynomial is real stable if and only if it is real rooted (as defined at the beginning of Section~\ref{sec:interlacing}).

One of the classical theorems in this area gives a direct link between interlacing and real stability.
\begin{theorem}[Hermite-Biehler]\label{thm:HB}
Let $f$ and $g$ be polynomials with real coefficients.
Then $g + i f$ is stable if and only if $f$ and $g$ have all real roots and $f \ll g$.
\end{theorem}

To prove that the polynomials we construct in this paper are real stable,
  we begin with an observation of Borcea and \Branden~\cite[Proposition
  2.4]{BBjohnson}.
\begin{proposition}\label{pro:BBdet}
If $A_{1}, \dots , A_{m}$
  are positive semidefinite Hermitian matrices, then the polynomial
\[
  \mydet{\sum_{i} z_{i} A_{i}}
\]
is real stable.
\end{proposition}

We will generate new real stable polynomials from the one above by applying operators of the form
  $(1 - \partial_{z_{i}})$.
One can use general results, such as Theorem~1.3 of \cite{BBWeylAlgebra}
  or Proposition~2.2 of \cite{LiebSokal}, to prove that these operators
  preserve real stability.
It is also easy to prove it directly using the fact that the analogous operator
  on univariate polynomials preserves stability of polynomials with complex coefficients.
For example, the following theorem appears as Corollary 18.2a in Marden~\cite{Marden},
  and is similar to Corollary 5.4.1 of Rahman and Schmeisser \cite{rahmanSchmeisser}.

\begin{theorem}\label{thm:marden}
If all the zeros of a degree $d$ polynomial $q (z)$ lie in a (closed) circular region $A$,
  then for $\lambda \in \Complex{}$, all the zeros of
\[
  q (z) - \lambda q' (z)
\]
lie in the convex region swept out by translating $A$ in the magnitude and direction
  of the vector $d \lambda$.
\end{theorem}

\begin{corollary}\label{cor:partialRealStable}
If $p \in \R[z_{1}, \dots , z_{m}]$ is real stable,
  then so is
\[
  (1 - \lambda\partial_{z_{1}}) p (z_{1}, \dots , z_{m}).
\]
for any $\lambda \in \R$.
\end{corollary}
\begin{proof}
Let $x_{2}, \dots , x_{m}$ be numbers with positive imaginary part.
Then, the univariate polynomial
\[
  q (z_{1}) =  p (z_{1}, z_{2}, \dots , z_{m}) \big|_{z_{2} = x_{2}, \dots , z_{m} = x_{m}}
\]
is stable.
That is, all of its zeros lie in the circular region consisting of numbers with non-positive
  imaginary part.
As this region is invariant under translation by $d$, $(1-\lambda \partial_{z_{1}}) q (z)$
  is stable.
This implies that $(1- \lambda\partial_{z_{1}}) p$ has no roots in which all of the variables have positive
  imaginary part.
\end{proof}

We will also use the fact that real stability is preserved under setting
  variables to real numbers (see, for instance, \cite[Lemma 2.4(d)]{wagner}).
\begin{proposition} \label{prop:setreal}
If $p\in\R[z_1,\ldots,z_m]$ is real stable and $a\in \R$, then
$p|_{z_1=a}=p(a,z_2,\ldots,z_m)\in\R[z_2,\ldots,z_m]$ is real stable.
\end{proposition}

Now it is immediate from Proposition~\ref{pro:BBdet} and Corollary~\ref{cor:partialRealStable} 
  that the mixed characteristic polynomial is real rooted.

\begin{corollary}\label{cor:mixedStable}
The mixed characteristic polynomial of positive semidefinite Hermitian matrices
  is real rooted.
\end{corollary}
\begin{proof}
Proposition~\ref{pro:BBdet} tells us that 
\[
\mydet{xI+\sum_{i=1}^m z_{i} A_{i}}
\]
  is real stable.
Corollary~\ref{cor:partialRealStable} tells us that
\[
  \left(\prod_{i=1}^{m} 1 - \partial_{z_{i}} \right)
  \mydet{xI+\sum_{i=1}^{m} z_{i} A_{i}}
\]
is real stable as well.
Finally, Proposition \ref{prop:setreal} shows that setting all of the $z_i$ to
  zero preserves real stability.
As the resulting polynomial is univariate, it is real rooted.
\end{proof}

Finally, we use the real rootedness of mixed characteristic polynomials to 
show that every sequence of independent finitely supported random
vectors $v_1,\ldots, v_m$ defines an interlacing family.
For $i  \in [m]$, let $\ell_i$ be the size of the support of $v_i$, and let
\[
\prob{}{v_i = w_{i,j}} = p_{i,j}
\]
for $j = 1, \dots, \ell_i$.

For a vector $s \in [l_{1}] \times \dots \times [l_{m}]$, we define
\[
  q_{s}(x) =
  \left(\prod_{i=1}^{m} p_{i, s_i} \right)
  \charp{\sum_{i=1}^{m} w_{i,s_i} w_{i,s_i}^{*} }{x}.
\]
\begin{theorem}\label{thm:mixedInterlacing}
The polynomials $q_{s}$ form an interlacing family.
\end{theorem}

\begin{proof}
For $0 \leq k < m$ and $t \in [l_{1}] \times \dots \times [l_{k}]$, we will write the conditionally expected polynomials
\[
q_{t} (x)
 = 
\left(\prod_{i=1}^{k} p_{i,t_i} \right)
\expec{v_{k+1}, \dots , v_{m}}{\charp{\sum_{i=1}^{k} w_{i,t_i} w_{i,t_i}^{*} + \sum_{j=k+1}^{m} v_{j} v_{j}^{*} }{x}}.
\]
In particular,
\[
q_{\emptyset} (x) = \expec{v_{1}, \dots , v_{m}}{\charp{\sum_{j=1}^{m} v_{j} v_{j}^{*} }{x}}.
\]
is the expected characteristic polynomial of the random matrix appearing in
Theorem\ref{thm:general}.
For a given $t = t_1, \dots, t_k$ and a given $r \in \ell_{k+1}$ let $(t,r)$ denote the vector $t_1, \dots, t_k, r$.
In this language we need to prove that for every $t$, the polynomials
$\{ q_{(t,r)} (x) : r \in \ell_{k+1} \}$ have a common interlacing.
By Lemma~\ref{lem:Fisk}, it suffices to prove that for any choice of real numbers $\{ \alpha_r \}$ with $0 \leq \alpha_r \leq 1$ and $\sum_r \alpha_r = 1$, the polynomial
\begin{equation}\label{eqn:convex_combo}
\sum_{r \in \ell_{k+1}} \alpha_r q_{(t, r)} (x)
\end{equation}
is real-rooted. 
But notice that 
\begin{align*}
q_{t} (x)
&=  
\left(\prod_{i=1}^{k} p_{i,t_i} \right)
\expec{v_{k+1}, \dots , v_{m}}{\charp{\sum_{i=1}^{k} w_{i,t_i} w_{i,t_i}^{*} + \sum_{j=k+1}^{m} v_{j} v_{j}^{*} }{x}}.
\\&= 
\sum_{r \in \ell_{k+1}} \left(\prod_{i=1}^{k} p_{i,t_i} \right) p_{k+1, r}
\expec{v_{k+2}, \dots , v_{m}}{\charp{\sum_{i=1}^{k-1} w_{i,j} w_{i,j}^{*} + w_{k+1,r}w_{k+1,r}^{*}+ \sum_{i=k+2}^{m} v_{i} v_{i}^{*} }{x}}
\\&= 
\sum_{r \in \ell_{k+1}} p_{k+1, r} q_{(t, r)}(x)
\end{align*}
which (if we set $\alpha_r = p_{k+1,r}$) is precisely the polynomial in (\ref{eqn:convex_combo}).
Thus it suffices to show that $q_t$ is real-rooted (independent of the values of $p_{k+1, r}$).
Denoting $\expec{}{v_i v_i^*} = A_i$, we have that for $t = t_1, \dots, t_k$,
\[
q_t(x) = \left(\prod_{i=1}^{k} p_{i, t_i} \right)\mixed{w_{1, t_1}w_{1, t_1}^*, \dots, w_{k, t_k}w_{k, t_k}^*, A_{k+1}, \dots , A_{m}}{x},
\]
a multiple of a mixed characteristic polynomial.
But by Corollary~\ref{cor:mixedStable}, such a polynomial is real-rooted regardless of what $A_{k+1} = \expec{}{v_{k+1} v_{k+1}^*}$ is, and therefore is real-rooted independent of the distribution on $v_k$, as needed.
\end{proof}

\section{The Multivariate Barrier Argument}\label{sec:barrier}

Our goal in this section is to prove an upper bound on the roots of the mixed
 characteristic polynomial $\mixed{A_1, \dots, A_m}{x}$ as a function of the
 $A_i$, in the case of interest $\sum_{i=1}^mA_i = I$.
Our main theorem is:
\begin{theorem}\label{thm:mixedbound} 
Let $A_1,\ldots,A_m$ be positive semidefinite Hermitian matrices satisfying 
\[
 \sum_{i=1}^m A_i = I
 \AND
 \trace{A_i}\leq \epsilon
\]
for all $i$. 
Then the largest root of $\mixed{A_1,\ldots,A_m}{x}$ is at most $(1+\sqrt{\epsilon})^2$.
\end{theorem}

We begin by performing a simple but useful change of variables that will allow
  us to reason separately about the effect of each $A_i$ on
  the roots of $\mixed{A_1,\ldots,A_m}{x}$.
\begin{lemma}\label{lem:mixedAlt}
Let $A_{1}, \dots , A_{m}$ be Hermitian positive semidefinite matrices.
If $\sum_{i} A_{i} = I$,
 then
\begin{equation}\label{eqn:mixed2}
  \mixed{A_{1}, \dots , A_{m}}{x} = 
  \left(\prod_{i=1}^{m} 1 - \partial_{y_{i}} \right)
  \mydet{\sum_{i=1}^{m} y_{i} A_{i}}
  \Big|_{y_{1} = \dots = y_{m} = x}.
\end{equation}
\end{lemma}
\begin{proof}
For any differentiable function $f$, we have
\[
  \partial_{y_{i}} (f (y_{i})) \big|_{y_{i} = z_{i} + x}
 = 
  \partial_{z_{i}} f (z_{i} + x).
\]
So, the lemma follows by substituting
  $y_{i} = z_{i} + x$
  into the expression \eqref{eqn:mixed2},
  and observing that it produces the expression
  on the right hand side of \eqref{eqn:mixed1}.
\end{proof}

Let us write
\begin{equation}\label{eqn:plugq} \mixed{A_1,\ldots,A_m}{x} =
Q(x,x,\ldots,x),\end{equation}
 where $Q(y_1,\ldots,y_m)$ is the multivariate polynomial on the right hand side of
 \eqref{eqn:mixed2}.
The bound on the roots of $\mixed{A_1,\ldots,A_m}{x}$ will follow from a
 ``multivariate upper bound'' on the roots of $Q$, defined as follows.
\begin{definition} Let $p(z_1,\ldots,z_m)$ be a multivariate polynomial. We say
that $z\in\R^m$ is {\em above} the roots of $p$ if 
\[
p(z+t)>0\qquad\textrm{for all}\qquad t=(t_1,\ldots,t_m)\in\R^m, t_i\ge 0,
\]
i.e., if $p$ is positive on the nonnegative orthant with origin at $z$.
\end{definition}
We will denote the set of points which are above the roots of $p$ by
$\above{p}$ (for convenience, we will say that $\above{0} = \R^m$).
A simple lemma we will find useful is that the region above the roots of $p$ never shrinks under the operation of partial differentiation.
\begin{lemma}
For any real stable polynomial $p$, $\above{p} \subseteq \above{p_{z_i}}$.
\end{lemma}

To prove Theorem~\ref{thm:mixedbound}, it is sufficient by
\eqref{eqn:plugq} to show that
  $(1+\sqrt{\epsilon})^2\cdot\bvec{1}\in\above{Q}$, where $\bvec{1}$ is the all-ones vector.
We will achieve this by an inductive ``barrier function'' argument.
In particular, we will construct $Q$ iteratively via a sequence of operations of
  the form $(1-\partial_{y_i})$, and we will track the locations of the roots of
  the polynomials that arise in this process by studying the evolution of the 
  functions defined below.

Our procedure will be to transform the real stable polynomial
\[
p(z_1, \dots, z_n) =  \mydet{z_i \sum_i A_i}
\]
into 
\[
p_n(x) = \prod_i (1- \partial_i)   \mydet{z_i \sum_i A_i}
\]
iteratively, keeping track of what happens to the points above the roots of $p$.
At the beginning, the region above the roots will be the positive orthant (since the $A_i$ are all positive semidefinite).
At step $k$, we will performing the $(1-\partial_k)$ operation to the polynomial, which (for lack of a better analogy) one can think of as a hitting a metal cast of the ``above the roots'' region with a hammer.
In particular, it will do two things:
\begin{enumerate}
  \item Shift the entire region in the $z_k$ direction
  \item Cause the region to flatten inwards (in all directions)
\end{enumerate}
Both effects will cause the zeros of the polynomial to move away from the origin, but the goal will be to bound the amount of movement.
Our method of obtaining such a bound uses a collection of measurements that tell us how convex the polynomial is at a given point, in a given variable.
We call these measurements ``barrier functions'' and we will need one such measurement for each variable in $p$.

\begin{definition} Given a real stable polynomial $p$ and a point $z=(z_1,\ldots,z_m)\in\above{p}$, the {\em barrier function of $p$ in direction $i$ at $z$} is defined as
\[
\Phi^i_p(z) = \frac{\partial_{z_i}p(z)}{p(z)}.
\]
\end{definition}
\noindent Equivalently, we may define $\Phi^i_p$ as 
\begin{equation}\label{eqn:concdef}
\Phi^i_p(z_1,\ldots,z_m) := \frac{q_{z,i}'(z_i)}{q_{z,i}(z_i)} = \sum_{j=1}^r\frac{1}{z_i-\lambda_j},
\end{equation}
where the univariate restriction
\begin{equation}\label{eqn:restrictdef}
q_{z,i}(t) := p(z_1,\ldots,z_{i-1},t,z_{i+1},\ldots,z_m)
\end{equation}
has roots $\lambda_1,\ldots,\lambda_r$ (which are all real, by Proposition \ref{prop:setreal}).

Note that the barrier functions are (for general points) not particularly well behaved, but we will only be considering them on the set of points that are above the roots, where they have a number of nice properties that we will exploit.
Of course applying the $(1-\partial_k)$ operation will have an effect on the values of the barrier functions, and so we will need to mindful of this change as well.
One observation that will simplify this is that the effect of applying $(1-\partial_k)$ to the barrier function $\Phi^j$ can be calculated with all of the other variables (not $z_k$ or $z_j$) fixed.
Hence it suffices to understand the effects on bivariate polynomials, an example of which is shown in Figure~\ref{fig:bivariate}.

\begin{figure}[h]
\begin{center}
\begin{tikzpicture}
\node (pic) at (0, 0) {\includegraphics[scale=1.4, trim={15 15 0 0}, clip]{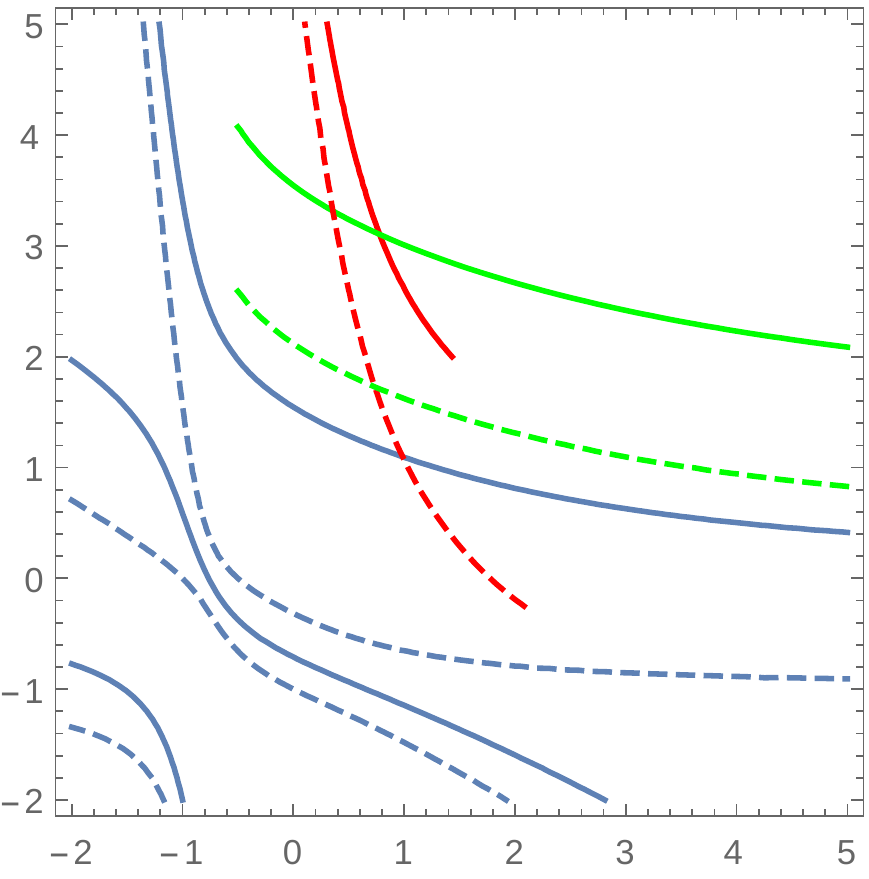}};
\node (legend) at (3.5, 3.5) {
  \def\arraystretch{1.25}%
  \begin{tabular}{rl}
   {\color{blue} - -} &  $p(x, y) = 0$  \\
   {\color{red} - -} & $\Phi^x_p(x, y) = 0.8$ \\
   {\color{green} - -} & $\Phi^y_p(x, y) = 0.8$ \\
   {\color{blue} ---} & $q(x, y) = 0$  \\
   {\color{red} ---} & $\Phi^x_q(x, y) = 0.8$\\
   {\color{green} ---} & $\Phi^y_q(x, y) = 0.8$
  \end{tabular}    
  \def\arraystretch{1.0}%
};
\end{tikzpicture}
\end{center}
\vspace{-0.75cm}
\caption{The effect of the $(1 - \partial_y)$ operator on the bivariate real stable polynomial
$p(x, y) = 4 + 12 x + 8 x^2 + 17 y + 29 x y + 8 x^2 y + 14 y^2 + 13 x y^2 + y^3$.  
Note $q(x, y) = (1 - \partial_y) p(x, y)$.
}
\label{fig:bivariate}

\end{figure}

Our proof of these will use an an observation of Terry Tao that uses a characterization of interlacing polynomials that appears in \cite{wagner}.

\begin{lemma}\label{lem:characterization}
Let $f$ and $g$ be real rooted polynomials with leading coefficient having the same sign such that $f \ll g$.
Then 
\[
(-1)^k \frac{\partial^k}{(\partial x)^k} \frac{f(x)}{g(x)} \bigg|_{x = y} \geq 0
\]
for all $y \in \above{g}$.
\end{lemma}
\begin{proof}
Let $\beta_1, \dots \beta_n$ be the roots of $g$.
Note that the equation
\[
f(x) = g(x) \left(s + \sum_i \frac{t_i}{x - \beta_i} \right)
\]
defines $n+1$ linearly independent equations in $n+1$ variables (one for each coefficient) and therefore has a solution.
Furthermore, one can check that each $t_i$ is nonnegative by noting that $f(\beta_i) = t_i g'(\beta_i)$ and using the fact that both $f(\beta_i)$ and $g'(\beta_i)$ alternate between nonpositive and nonnegative values (the first is due to the interlacing, the second is always true).
The result then follows by taking the derivatives and using the fact that $y - \beta_i > 0$ for all $i$.
\end{proof}

Using this, we can show the two analytic properties of barrier functions that we need: at any point above the roots of a real stable polynomial, the barrier functions are nonincreasing and convex in every coordinate.
\begin{lemma}\label{lem:monotone}
Suppose $p$ is real stable and
$z\in\above{p}$. 
Then for all $i,j \leq m$ and $\delta \geq 0$,
\begin{align}
\label{eqn:mono}\Phi^i_p(z+\delta e_j) &\leq \Phi^i_p(z), \text{ and} \qquad&\textrm{(monotonicity)} \\
\label{eqn:conv}\Phi^i_p(z+\delta e_j) &\leq \Phi^i_p(z) + \delta\cdot\partial_{z_j}\Phi^i_p(z+\delta e_j) 
&\qquad\textrm{(convexity).}
\end{align}
\end{lemma}
\begin{proof} 
If $i=j$ then consider the real-rooted univariate restriction $q(x_i) =
\prod_{k=1}^r(x_i - \lambda_k)$ defined in
\eqref{eqn:restrictdef}. 
Since $z\in\above{p}$ we know that $z_i>\lambda_k$ for all $k$. 
Monotonicity follows immediately by
considering each term in \eqref{eqn:concdef}, and convexity is easily
established by computing
\[
\partial^2_{x_i}\left(\frac{1}{x_i-\lambda_k}\right) \bigg|_{x=z}
\frac{2}{(z_i-\lambda_k)^3}
\]
which is positive since (for $z \in \above{p}$) $z_i >\lambda_k$.
For the case when $i \neq j$, we fix all variables other than $x_i$ and $x_j$ and consider the
bivariate restriction
\[
q(x_i, x_j) := p(z_1, \dots, x_i, \dots, x_j, \dots, z_m).
\]
Using Lemma~\ref{lem:characterization}, both monotonicity and convexity would follow by showing that $f \ll g$ where
\[
f(x_j) := \partial_{x_i} q(z_i, x_j)
\AND
g(x_j) := q(z_i, x_j)
\]
By Corollary~\ref{cor:partialRealStable}, $(1 + \lambda \partial_{x_i}) q$ is real stable (and so $g + \lambda f$ is real rooted) for all $\lambda$.
Hence by Theorem~\ref{thm:HKO} either $f \ll g$ or $g \ll f$.

To show that, in fact, $f \ll g$, we will consider the sum of the roots.
That is, it suffices to show that the sum of the roots of $f$ is at most the sum of the roots of $g$.
Write 
\[
q(x, y) = a(y)x^n + b(y)x^{n-1} + \dots
\]
Since taking partial derivatives preserves real stability, 
\[
\partial_{x}^{n-1} q(x,y) = (n-1)! (n x a(y) + b(y))
\]
is real stable.
Hence $b(y) + i a(y)$ is stable and so by Theorem~\ref{thm:HB}, we have $a \ll b$.
Using Lemma~\ref{lem:characterization} again, this implies $a(y) b'(y) -  b(y) a'(y) \geq 0$.

Now since $(x, y)$ is above the roots of $q$, it is also above the roots of $\partial_{x}^{n} q(x,y) = n! a(y)$ and so $a'(y)$ and $a(y)$ have the same sign.
Hence
\begin{equation}\label{eqn:sums}
\frac{b'(y)}{a'(y)} \geq  \frac{b(y)}{a(y)}
\end{equation}
Note that the sum of the roots of $g$ is $-b/a$ and the sum of the roots of $f$ is $-b'/a'$.
Thus (\ref{eqn:sums}) is asserting that the sum of the roots of $f$ is at most the sum of the roots of $g$ (as needed).
\end{proof}

\begin{remark} \label{rmk:renegar}
Our original proof of monotonicity and convexity used a powerful characterization of bivariate real stable polynomials due to Helton and Vinnikov~\cite{HeltonVinnikov} and Lewis, Parrilo and Ramana~\cite{lax}.
While this characterization is extremely useful, it (incorrectly) gave the the impression that such a powerful result was required to prove Lemma~\ref{lem:monotone}.
James Renegar, in particular, pointed out that the lemma follows directly from well-known properties of hyperbolic polynomials.
We chose the proof given here since it has the benefit of remaining in the domain of real stable polynomials.
\end{remark}

Our first observation is that when a point above the roots has a small enough boundary function in a given direction, it remains above the roots after applying an operator in that direction.
Pictorially, this asserts that the dotted green line in Figure~\ref{fig:bivariate} will always be contained inside the solid blue line.

\begin{lemma}\label{lem:above} 
Let $p$ be a real stable polynomial, and let $z$ be a point above the roots of $p$ which satisfies $\Phi_p^i(z)<1$. Then $z$ is also above the roots of $p-\partial_{z_i}p$.
\end{lemma}
\begin{proof}
Let $t$ be a nonnegative vector. As $\Phi$ is nonincreasing in
each coordinate we have $\Phi_p^i(z+t)<1$, whence
\[
\partial_{z_i}p(z+t)<p(z+t) \implies (p-\partial_{z_i}p)(x+t)>0,
\]
as desired.
\end{proof}

While Lemma~\ref{lem:above} proves what we need for a single iteration, it is not strong enough for an inductive argument because the application of a $(1-\partial_{z_k})$ operator will typically cause all of the barrier functions to increase.
As previously mentioned, the effect of the $(1-\partial_{z_k})$ operator will be to shift in the $z_k$ direction and flatten away from the origin.
To remedy this, we will translate our upper bounds in the $z_k$ direction as well (see Figure~\ref{fig:move}).
Certainly this will compensate for the shift in the $z_k$ direction, but we will need to move extra in order to compensate for the flattening of the region.
How much extra will be determined by the value of the barrier function in that direction.
In particular, by exploiting the convexity properties of the $\Phi_p^i$, we arrive at the following useful strengthening of Lemma~\ref{lem:above}.
\begin{lemma}\label{lem:barrier}
Suppose $p(z_1,\ldots,z_m)$ is real stable with $z\in\above{p}$, and $\delta>0$ satisfies
\begin{equation}\label{eqn:deltaCond}
\Phi^j_p(z) \leq 1 - \frac{1}{\delta}.
\end{equation}
Then for all $i$, 
\[
\Phi^i_{p-\partial_{z_j}p}(z+\delta e_j)\leq\Phi^i_p(z).
\]
\end{lemma}
The proof follows directly from property \eqref{eqn:conv} of Lemma~\ref{lem:monotone}.
We refer the reader to \cite{IF2} for the details.
The effect of Lemma~\ref{lem:barrier} can be seen in Figure~\ref{fig:move}.
By moving far enough in the $y$ direction, the given point is able to move back inside the regions defined by the $\Phi^x$ and $\Phi^y$ functions.

\begin{figure}[h]
\begin{center}
\begin{tikzpicture}
\node (pic) at (0, 0) {\includegraphics[scale=1.0, trim={15 15 0 0}, clip]{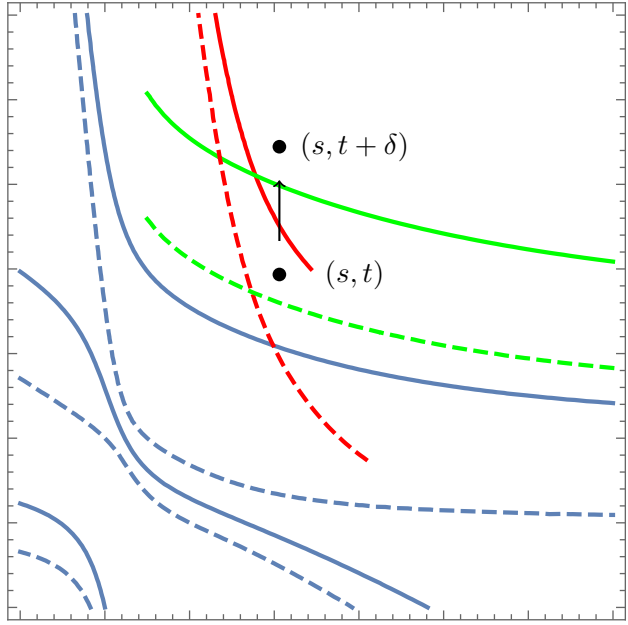}};

\node[circle, fill, scale=0.5] (x)  at (-0.5, 0.5) {};
\node[circle, fill, scale=0.5] (y) at (-0.5, 2.2) {};
\node [right of=y] (yl) {$(s, t + \delta)$ };
\node [right of=x] (xl) {$(s, t)$ };

\draw[shorten >=10pt,shorten <=10pt,->, thick] (x) -- (y);
\end{tikzpicture}
\end{center}
\vspace{-0.75cm}
\caption{Moving $\delta = \frac{1}{1-\Phi^y_p(s, t)}$ in the $y$ direction moves the point back inside the region defined by the two barrier functions.}
\label{fig:move}
\end{figure}

It should now be clear how the proof proceeds --- at each step, we will apply the operator $(1-\partial_k)$ and then move our upper bound in that direction.  
We will bound the amount we move using the barrier function in that direction, while also taking care that we have moved far enough to cause the barrier functions in all of the other directions to go down (so that they will still be small when the time comes to use them).

\begin{proof}[Proof of Theorem~\ref{thm:mixedbound}]
Let 
\[
  P (y_{1}, \dots , y_{m}) = \mydet{\sum_{i=1}^{m} y_{i} A_{i}}.
\]
Set $t = \sqrt{\epsilon} + \epsilon$.
As all of the matrices $A_{i}$ are positive semidefinite
  and
\[
  \mydet{t \sum_{i} A_{i}} = \mydet{t I} > 0,
\]
the vector $t \bvec{1}$ is above the roots of $P$.

By Theorem~\ref{thm:jacobi},
\[
  \Phi^{i}_{P} (y_{1}, \dots , y_{m})
=
  \frac{\partial_{i} P (y_{1}, \dots , y_{m})}{P (y_{1}, \dots , y_{m})}
=
\trace{\left(\sum_{i=1}^{m} y_{i} A_{i} \right)^{-1} A_{i}}.
\]
So,
\[
  \Phi^{i}_{P} (t \bvec{1})
=
  \trace{A_{i}} / t
\leq 
  \epsilon / t
= 
  \epsilon / (\epsilon + \sqrt{\epsilon}),
\]
which we define to be $\phi$.
Set
\[
  \delta = 1/ (1-\phi) = 1 + \sqrt{\epsilon }.
\]

For $k \in [m]$, define
\[
  P_{k} (y_{1}, \dots , y_{m})
=
  \left(\prod_{i=1}^{k} 1 - \partial_{y_{i}} \right) P (y_{1}, \dots , y_{m}).
\]
Note that $P_{m} = Q$.

Set $x^{0}$ to be the all-$t$ vector, and 
  for $k \in [m]$ define
  $x^{k}$ to be the vector that is $t+\delta$
 in the first $k$ coordinates and $t$ in the rest.
By inductively applying Lemmas~\ref{lem:above} and \ref{lem:barrier},
  we prove that $x^{k}$ is above the roots of $P_{k}$,
  and that for all $i$
\[
   \Phi^{i}_{P_{k}} (x^{k}) \leq \phi .
\]

It follows that the largest root of
\[
  \mixed{A_{1}, \dots , A_{m}}{x}
= 
  P_{m} (x, \dots , x)
\]
is at most 
\[
t + \delta =
  1 + \sqrt{\epsilon} + \sqrt{\epsilon} +  \epsilon  = (1+\sqrt{\epsilon})^{2}.
\]
\end{proof}

\section{Theorems}\label{sec:theorems}

We now combine the results in the previous sections to prove
Conjecture~\ref{conj:weaver}, thereby proving Conjecture \ref{conj:fpc} and
showing that the Kadison-Singer Problem has a positive solution.

We first complete the proof of Theorem \ref{thm:general}, restated here for
convenience.
\begin{theorem}[Theorem \ref{thm:general}]
Let $\epsilon > 0$ and let
  $v_1, \dots, v_m$ be independent random vectors in $\C^d$ with finite support such that 
\begin{equation}\label{eqn:mainSum2}
\sum_{i=1}^m \expec{}{v_{i} v_{i}^{*}} = I_d,
\end{equation}
and
\[
\expec{}{ \norm{v_{i}}^{2}} \leq \epsilon, \text{for all $i$.}
\]
Then 
\[
\prob{}{ \norm{\sum_{i=1}^m v_i v_i^{*}} \leq (1 + \sqrt{\epsilon})^2 } > 0
\]
\end{theorem}
\begin{proof}
Let $A_{i} = \expec{}{v_{i} v_{i}^{*}}$.
We have 
\[
\trace{A_{i}} = \expec{}{\trace{v_{i} v_{i}^{*}}}
  = \expec{}{v_{i}^{*} v_{i}} = \expec{}{\norm{v_{i}}^{2}} \leq \epsilon ,
\]
for all $i$.

The expected characteristic polynomial of the $\sum_{i} v_{i} v_{i}^{*}$
  is the mixed characteristic polynomial $\mixed{A_{1}, \dots , A_{m}}{x}$.
Theorem~\ref{thm:mixedbound} implies that the largest root of this polynomial 
  is at most $(1+\sqrt{\epsilon})^{2}$.

For $i \in [m]$, let $l_{i}$ be the size of the support of the random vector $v_{i}$,
  and let
  $v_{i}$ take the values $w_{i,1}, \dots , w_{i,l_{i}}$ 
  with probabilities
  $p_{i,1}, \dots , p_{i,l_{i}}$.
Theorem~\ref{thm:mixedInterlacing} tells us that the polynomials
  $q_{j_{1}, \dots , j_{m}}$ are an interlacing family.
So, Theorem~\ref{thm:interlacing} implies that there exist
  $j_{1}, \dots , j_{m}$ so that
  the largest root of the characteristic polynomial of
\[
  \sum_{i=1}^{m} w_{i,j_{i}} w_{i,j_{i}}^{*}
\]
is at most $(1 + \sqrt{\epsilon})^{2}$.
\end{proof}

It is worth noting here that the bound $(1 + \sqrt{\epsilon})^2$ is asymptotically tight, as can be seen by picking random vectors so that 
\[
  \expec{}{v_iv_i^*} = \frac{1}{n}I
\]
for all $i$.
The resulting polynomial is an associated Laguerre polynomial whose largest root is (asymptotically) exactly this bound.
We believe this polynomial is actually the extremal polynomial for this problem, but are unable to prove it.

Since the outer products of the vectors sum to the identity, the best one could hope for is to be able to split the vectors into $r$ groups such that each was {\em exactly} $(1/r)I$.
Hence Theorem~\ref{thm:partition} guarantees that for any vectors $v_i$, one can get within a factor of $\left(1+\sqrt{r \delta}\right)^{2}$ of the best one could get with the best possible $v_i$.

\begin{theorem}
Let $r > 0$ be an integer, and let $u_1, \dots, u_m \in \C^d$ be vectors such that 
\[
  \sum_{i=1}^m \expec{}{u_{i} u_{i}^{*}} = I_d
  \AND
\norm{u_i}^2 \leq \delta \text{ for all $i$.}
\]
Then there exists a partition $\{ A_1, \dots A_r \}$ of $[m]$ such that 
\[
\norm{ \sum_{i \in A_j} {u_i} {u_i}^{*}} \leq \frac{1}{r}\left(1+\sqrt{r \delta}\right)^{2}
\]
\end{theorem}
\begin{proof}
For each $i \in [m]$ and $k \in [r]$, define $w_{i,k} \in \C^{rd}$ to be the direct sum of $r$ vectors from $\C^d$, all of which are $0^d$ (the $0$-vector in $\C^d$) except for the $k^{th}$ one (which is a copy of $u_i$).
Now let $v_1, \dots, v_m$ be independent random vectors such that $v_i$ takes the values $\{ \sqrt{r} w_{i,k} \}_{k=1}^r$ each with probability $1/r$.

These vectors satisfy
\[
  \expec{}{v_{i} v_{i}^{*}}
= 
\begin{pmatrix}
u_{i} u_{i}^{*} & 0_{d \times d} & \dots & 0_{d \times d} \\
0_{d \times d} & u_{i} u_{i}^{*} & \dots & 0_{d \times d} \\
\vdots &  & \ddots &  \vdots \\
0_{d \times d} & 0_{d \times d} & \dots & u_{i} u_{i}^{*},
\end{pmatrix}
\quad \text{and} \quad 
\norm{v_{i}}^{2} = r \norm{u_{i}}^{2} \leq r \delta.
\]
So,
\[
\sum_{i=1}^m  \expec{}{v_{i} v_{i}^{*}} = I_{rd}
\]
and we can apply Theorem~\ref{thm:general} with $\epsilon  = r \delta$ to show that 
  there exists an assignment of each $v_i$ so that  
\[
\norm{
\sum_{k=1}^r
\sum_{i : v_i = w_{i,k}}
\left(\sqrt{r} w_{i,k} \right)
\left(\sqrt{r} w_{i,k} \right)^{*}
}
\leq (1+\sqrt{r \delta})^{2}.
\]
Setting $A_{k} = \{ i : v_i = w_{i,k} \}$ implies that 
\[
\norm{\sum_{i \in A_k} {u_i} {u_i}^{*}} 
= 
\norm{\sum_{i \in A_k} {w_{i,k}} {w_{i,k}}^{*}} 
\leq  
\frac{1}{r}
\norm{
\sum_{k=1}^r
\sum_{i : v_i = w_{i,k}}
\left(\sqrt{r} w_{i,k} \right)
\left(\sqrt{r} w_{i,k}\right)^{*}
}
\leq \left(\frac{1}{\sqrt{r}}+\sqrt{\delta}\right)^{2}.
\]
and this is true for all $k$.
\end{proof}

Note that the bound in Theorem~\ref{thm:partition} is actually quite good.
Since the outer products of the vectors sum to the identity, the best one could hope for is to be able to split the vectors into $r$ groups such that each was {\em exactly} $(1/r)I$.
Hence Theorem~\ref{thm:partition} guarantees that for any vectors $v_i$, one can get within a factor of $\left(1+\sqrt{r \delta}\right)^{2}$ of the best one could get with the best possible $v_i$.

\section{Extensions}\label{sec:extensions}
There are multiple issues that occur when attempting to apply
Theorem~\ref{thm:general} in the context of more general matrices.
Two of these issues have since been resolved, and we list the resulting theorems here without proof.
The first is a result due to Michael Cohen \cite{cohen} that allows one to get bounds in situations where the original matrices are not rank 1 \cite{cohen}.

\begin{theorem}\label{thm:cohen}
Let $\epsilon > 0$ and let $A_1, \dots, A_m$ be independent random positive semidefinite matrices in $\C^{d \ times d}$ with finite support such that 
\[
\sum_{i=1}^m \expec{}{A_i} = I_d
\AND
\expec{}{ \trace{A_i} } \leq \epsilon, \text{for all $i$.}
\]
Then 
\[
\prob{}{ \norm{\sum_{i=1}^m A_i} \leq (1 + \sqrt{\epsilon})^2 } > 0.
\]
\end{theorem}
This then leads to a similar generalization of Theorem~\ref{thm:partition}.
The issue with trying to apply the original proof in the high rank setting is
that the expected characteristic polynomial $\E\det(xI-\sum_{i} A_i)$ need not
be real-rooted for random high-rank matrices $A$ --- the simplest example is
just to take $2\times 2$ $A_i$ equal to $I$ and $-I$ with equal probability,
yielding the expected polynomial:
$$ (x-1)^2+(x+1)^2,$$
which is strictly positive on the real line. Cohen's idea is to replace the
characteristic polynomial $\chi(\sum_{i}A_i)$ by the {\em mixed characteristic polynomial}:
$$\mu(A_1,\ldots,A_m):=\E\det(\sum_{i=1}^m B_i)$$
where the $B_i$ are any rank one random matrices such that $\E B_i(=A_i)$. This polynomial is necessarily real-rooted by Corollary
\ref{cor:mixedStable}, and by construction multilinear in the $A_i$, so the
same argument as in the rank-1 case shows that there exist $(A_1,\ldots,A_m$
such that the largest root of $\mu(A_1,\ldots,A_m)$ is at most the largest root
of $\E \mu(A_1,\ldots,A_m)$. The latter polynomial is just $\E\chi(xI-\sum_i
B_i)$, so its largest root is at most $(1+\sqrt{\epsilon})^2$ where
$\epsilon=\E\tr(B_i)=\E\tr(A_i)$. The main technical content of Cohen's result,
which is a kind of convexity result, is that 
$$\lambda_{max}\chi(\sum_{i\le m}A_i)\le \lambda_{max}\mu(A_1,\ldots,A_m)$$
for any positive semidefinite $A_i$. Combining this with the previous result
gives Theorem \ref{thm:cohen}.

It should be noted that this extension only works when one wishes to find a
polynomial whose {\em largest root} is {\em small}.  In general, the method of
interlacing polynomials will supply a bound on either side of any chosen root
(for example, to find a polynomial whose smallest root is large), but we
currently can only get such a bound in the rank 1 case.  Fortunately, the
majority of applications seem to involve the largest root, so for many of the
known results, the extension can typically be applied without issue.

One reaction to the high rank extension is to think that now one might be able
to find a partition of the type in Theorem~\ref{thm:partition} with some number
of added constraints (that one can impose combinatorially using added dimensions
orthogonal to the original vectors).  This, unfortunately, does not appear to be
as useful as one might think, since the addition will cause the expected trace
of the matrices to go up, thereby decreasing the accuracy of the bound.  This
suggests that a similar extension would not hold if one was to find an analogue
of Theorem~\ref{thm:general} for indefinite self adjoint matrices (which would
be interesting in its own right).  This does suggest as an interesting open
question whether one can find a bound analagous to Theorem~\ref{thm:general}
when the restriction on the matrices is expressed in some other (for example,
Schatten) norm.  

The other issue is with the constraint (\ref{eqn:mainSum2}).  This has been dealt
with in a paper of Akemann and Weaver \cite{akemann_weaver} that shows (among
other things) the following extension of Theorem~\ref{thm:partition}.

\begin{theorem}\label{thm:partition2}
Let $r > 0$ be an integer, and let $u_1, \dots, u_m \in \C^d$ be vectors such that 
\[
  \sum_{i=1}^m \expec{}{u_{i} u_{i}^{*}} \leq I_d
  \AND
\norm{u_i}^2 \leq \delta \text{ for all $i$.}
\]
Now let $t_1, \dots, t_n$ satisfy $0 \leq t_i \leq 1$.
Then there exists a partition $\{ A_1, \dots A_r \}$ of $[m]$ such that 
\[
\norm{ \sum_{i \in A_j} {u_i} {u_i}^{*} - \sum_{i} t_i {u_i} {u_i}^{*}} \leq O(\delta^{1/8}).
\]
\end{theorem}

Their method of proof proceeds by proving a weighted version of Theorem~\ref{thm:general} and then giving a series of successive approximations (see \cite{akemann_weaver} for details).

Lastly we mention that even though Theorem~\ref{thm:general} is asymptotically tight, further restrictions on the random vectors can lead to improved bounds.
This leads to slight improvements in the guarantee of Theorem~\ref{thm:partition} in cases where the partition size is small.
Looking at the proof of Theorem~\ref{thm:partition}, there is a direct correspondence between the number of partitions and the rank of the expected matrix that is constructed.
This in turn corresponds to the degree of the variable $z_i$ in the polynomial
\[
  \mydet{xI - \sum_i z_i A_i}.
\]
This degree restriction can sometimes lead to tighter bounds on the convexity of the barrier functions.
For example, the following improvement of Lemma~\ref{lem:barrier} was shown in \cite{improved_ks2}:
\begin{lemma}\label{lem:phi}
Assume $p(x, y)$ is quadratic in $x$ and let
\[
\Phi_p^x \leq \left(1 - \frac{1}{\delta} \right) \frac{1}{2 - \delta}
\]
for some $\delta \in (1, 2)$.
Now let $q(x, y) = (1 - \partial_x) p(x + \delta, y)$ and assume that $(x_0, y_0)$ is above the roots of \textbf{both} $p$ and $q$.
Then
\[
\Phi_{q}^y \leq \Phi_{p}^y.
\]
\end{lemma}

Using this, the authors then go on to show an improvement of Theorem~\ref{thm:partition} when $r = 2$ and $\delta < 1/2$:
\begin{theorem}\label{thm:partition3}
Let $u_1, \dots, u_m \in \C^d$ be vectors such that 
\[
  \sum_{i=1}^m \expec{}{u_{i} u_{i}^{*}} = I_d
  \AND
\norm{u_i}^2 \leq \delta \leq \frac{1}{2} \text{ for all $i$.}
\]
Then there exists a partition $\{ A_1, A_2 \}$ of $[m]$ such that 
\[
\norm{ \sum_{i \in A_j} {u_i} {u_i}^{*}} \leq \frac{1}{2} + \sqrt{\delta}\sqrt{1-\delta}.
\]
\end{theorem}

The argument in \cite{improved_ks2} follows the same general pattern as the one
used in our proof Theorem~\ref{thm:mixedbound}, but has a number of added
issues.  One of the nicer occurrences in the proof of
Theorem~\ref{thm:mixedbound} is that the the conditions necessary for
Lemma~\ref{lem:barrier} were strictly stronger than the conditions necessary for
Lemma~\ref{lem:above}.  As a result, the constraint provided by
Lemma~\ref{lem:barrier} is the only relevant one in determining the optimal values
of $\delta$ and $t$.  With the improved version of Lemma~\ref{lem:phi}, this is
no longer the case and so one must balance two competing constraints.  In the
end, Theorem~\ref{thm:partition3} only gives a slight improvement over the value
\[
  \frac{1}{2} + \sqrt{2\delta} + \delta
\]
that one gets directly from Theorem~\ref{thm:partition}.

\bibliographystyle{alpha}
\bibliography{CDM}

\end{document}